\documentclass[11pt]{article}
\input{amssymb.sty}
\usepackage{amsmath, amsthm}

\usepackage{hyperref} 
\usepackage{nameref}
\usepackage{makeidx}
\usepackage{enumerate}
\usepackage{stmaryrd}
\usepackage{mathrsfs}
\title{Maximal amenability and disjointness for the radial masa}
\author{Chenxu Wen}

\AtEndDocument{\bigskip{\footnotesize%
  \textsc{Department of Mathematics, Vanderbilt University, 1326 Stevenson Center, Nashville, TN 37240, United States} \par  
  \textit{E-mail address}: \texttt{chenxu.wen@vanderbilt.edu} \par
}}

\date{}

\newcommand\norm[1]{\left\lVert#1\right\rVert}
\newtheorem{thm}{Theorem}
\newtheorem*{thm*}{Theorem}
\newtheorem*{cor*}{Corollary}
\newtheorem*{conj*}{Conjecture}
\newtheorem{prop}[thm]{Proposition}

\newtheorem{lem}[thm]{Lemma}
\newtheorem*{defn*}{Definition}
\newtheorem*{que*}{Question}
\theoremstyle{definition}

\newtheorem{rem}[thm]{Remark}

\begin{document}

\maketitle

\begin{abstract}
We prove that the radial masa $C$ in the free group factor is disjoint from other maximal amenable subalgebras in the following sense: any distinct maximal amenable subalgebra cannot have diffuse intersection with $C$.
\end{abstract}

\section*{Introduction}
Amenability is one of the most central notions in the study of von Neumann algebras. Amenable von Neumann algebras are very well understood: Connes' work on the characterization of amenability \cite{connes76classification} is a milestone for the entire theory. Thus, in order to study non-amenable von Neumann algebras, it is natural to consider their amenable subalgebras.

Fuglede and Kadison \cite{fugledekadison51normalcy} showed that for any II$_1$ factor, there always exists a maximal hyperfinite subfactor, thus answered a question of Murray and von Neumann about the double relative commutant. Later on, during a conference at Baton Rouge in 1967, Kadison asked a series of famous questions about von Neumann algebras (see for example \cite{ge03kadisonquestions}). Among them is the following:

\begin{que*}
Is every self-adjoint element in a II$_1$ factor contained in a hyperfinite subfactor?
\end{que*}

Popa answered this question in the negative, by showing that the generator masa in the free group factor is maximal amenable, \cite{popa83maxinjective}. 

If $(M,\tau)$ is a finite von Neumann algebra with a faithful normal tracial state $\tau$ and $\omega$ is a free ultrafilter, we'll write $M^{\omega}$ as the ultraproduct of $(M,\tau)$. The key insight of Popa \cite{popa83maxinjective} is that the inclusion $A\subset M$, where $M=L(\mathbb{F}_n)$ with $n\geq 2$ and $A$ the generator masa, satisfies the \textit{asymptotic orthogonality property}, which we define below:

\begin{defn*}
Let $A\subset M$ be an inclusion of finite von Neumann algebras. We say that the inclusion satisfies the asymptotic orthogonality property (AOP for short), if for any free ultrafilter $\omega$ on $\mathbb{N}$, $(x_n)_n\in A'\cap M^{\omega}\ominus A^{\omega}$ and $y_1, y_2\in M\ominus A$, we have that $y_1(x_n)_n\perp (x_n)_ny_2$.
\end{defn*}

Since Popa, there are many results considering maximal amenable subalgebras. Ge \cite[Theorem 4.5]{ge96maxinjective} showed that any diffuse amenable finite von Neumann algebra can be realized as a maximal amenable subalgebra of the free group factor. Shen \cite{shen06tensor} showed that the $\bigotimes_{n\in \mathbb{N}} A$ is maximal amenable inside $\bigotimes_{n\in \mathbb{N}} M$, where $A$ is the generator masa in the free group factor $M$, thus gave an example of a maximal masa in a McDuff-II$_1$ factor. Cameron, Fang, Ravichandran and White \cite{CFRW2010radialmasa} proved that the radial masa in the free group factor is maximal amenable. Brothier \cite{arnaud14maxinjective} gave an example in the setting of planar algebras.  Boutonnet and Carderi \cite{boutonnet13hyperbolic} showed that the subalgebra coming from a maximal amenable subgroup in a hyperbolic group, is maximal amenable. Houdayer \cite{houdayer14exoticmaxinjective} showed that the factors coming from free Bogoljubov actions contains concrete maximal amenable masa's, see also \cite{houdayer14gammastability}. All these results use Popa's AOP approach. 

Very recently a new method via the study of centralizers of states, is developed by Boutonnet and Carderi \cite{boutonnet14centralizer}. In particular, they are able to show that the subalgebra coming from the upper-triangular matrix subgroup of $SL(3,\mathbb{Z})$, is maximal amenable inside $L(SL(3,\mathbb{Z}))$. See Ozawa's remark \cite{ozawa15centralizer} for an application of this new approach.

Another central theme in the theory is the study of the free group factors (\cite{ringsofoperators4}, \cite{freeentropyIII}, \cite{primenessforfreegroupfactors}, \cite{ozawa04solidity}, \cite{op10uniquecartan}). One of the motivating questions of this paper is a conjecture by J. Peterson (see the end of \cite{petersonthom11cocycle}):
\begin{conj*}
For the free group factor, any diffuse amenable subalgebra is contained in a \textit{unique} maximal amenable subalgebra.
\end{conj*}

Houdayer's result on Gamma stability of free products \cite[Theorem 4.1]{houdayer14gammastability} implies that the generator masa satisfies Peterson's conjecture. The proof again is relying on the AOP. See also Ozawa's proof \cite{ozawa15centralizer} via the centralizer approach.

One subalgebra of the free group factor under intense study is the radial masa. So let $M=L(\mathbb{F}_N)$ with $2\leq N<\infty$ be the free group factor with finitely many generators and denote by $C$ the von Neumann subalgebra of $M$ generated by $\omega_1:=\sum_{g\in \mathbb{F}_N, |g|=1}u_g$. Note that $\omega_1$ is only well-defined for free groups with finitely many generators. It was proved by Pytlik \cite[Theorem 4.3]{pytlik81} that $C$ is a masa in $M$, called the \textit{radial masa} or the \textit{Laplacian masa}. Moreover, R{\u a}dulescu \cite[Theorem 7]{radulescu91radialmasa} showed that $C$ is singular and Cameron, Fang, Ravichandra and White \cite[Corollary 6.3]{CFRW2010radialmasa} proved that it is maximal amenable in $M$. 

Recall that a result of Popa \cite[Corollary 4.3]{popa83orthogonal} shows that generator masa's coming from different generators cannot be unitarily conjugate inside $M$. This implies that the radial masa $C$ cannot be unitarily conjugate with the generator masa $A$ inside $M$. However, whether they are conjugate via some automorphism, is still unknown.

The main result of this paper is the following:

\begin{thm*}
Let $M=L(\mathbb{F}_N)$ with $2\leq N<\infty$ and let $C\subset M$ be the radial masa. Then every amenable subalgebra of $M$ having diffuse intersection with $C$, must be contained in $C$.  
\end{thm*}

This is the first example of such disjointness for an maximal amenable subalgebra which is not known to be in a free position. 

The approach taken in this paper is to show a stronger version of Popa's AOP. The main idea is that, according to a computation by R{\u a}dulescu \cite{radulescu91radialmasa}, $\ell^2(\mathbb{F}_N)\ominus L^2(C)$ admits a very nice decomposition as a direct sum of $C$-$C$-bimodules. Moreover, each bimodule contains a Riesz basis whose interaction with the left and right actions of $C$ is very similar to that of the canonical basis of the free group.  In other words, the radial masa $C$ and $L^2(M)\ominus L^2(C)$ behave almost freely.

\section*{Acknowledgement}
We would like to thank Jesse Peterson, for suggesting the question and for many valuable discussions and consistent encouragement. We are very grateful to Cyril Houdayer, R{\' e}mi Boutonnet and Stuart White for a careful reading and numerous helpful suggestions on an early version of the paper. We also thank Sorin Popa and Allan Sinclair for their comments. Last but not least, we would like to thank the anonymous referee for providing many helpful comments.

\section*{Proof of the Theorem}
Recall that a von neumann algebra $M$ is said to be \textit{solid}, if the relative commutant of any diffuse amenable subalgebra is amenable \cite{ozawa04solidity}. $M$ is called \textit{strongly solid}, if for any diffuse amenable subalgebra $D$ of $M$, the normalizer of $D$ generates an amenable subalgebra of $M$ \cite{op10uniquecartan}. Clearly strong solidity implies solidity. It is shown in \cite{op10uniquecartan} that free group factors are strongly solid.

 The following proposition is inspired by \cite[Lemma 3.1, Theorem 3.2]{popa83maxinjective} and \cite[Lemma 2.2, Corollary 2.3]{CFRW2010radialmasa}:
\begin{prop}\label{prop: s-AOP implies disjointness}
Let $M$ be a strongly solid II$_1$ factor and $A\subset M$ a singular masa in $M$. Assume in addition that for any diffuse von Neumann subalgebra $B\subset A$ and any free ulltrafilter $\omega$, the following holds:

for any $(x_k)_k\in B'\bigcap M^{\omega}\ominus A^{\omega}$ and for any $y_1,y_2\in M\ominus A$, we have that $y_1(x_k)_k\perp (x_k)_ky_2$.

Then any amenable subalegbra of $M$ containing $B$, must be contained in $A$.
\end{prop}
\begin{proof}

As shown by \cite[Lemma 2.2, Corollary 2.3]{CFRW2010radialmasa}, AOP and singularity imply that $A$ is maximal amenable in $M$.

Let $B\subset Q\subset M$ be an amenable subalgebra. By solidity of $M$, $A\subset B'\bigcap M$ is amenable. Since $A$ is maximal amenable, we conclude $Q'\bigcap Q\subset B'\bigcap M\subset A$.

Let $z$ be the maximal central projection of $Q$ such that $Qz$ is type II$_1$. Now suppose that $z\neq 0$.

Since $Qz$ is amenable and of type II$_1$, Popa's intertwining theorem (\cite[Theorem A.1]{popa03betti}) easily implies that there is a unitary $u\in (Qz)'\bigcap (Qz)^{\omega}$, such that $E_{A^{\omega}}(u)=0$. For a proof, see \cite[Lemma 2.2]{CFRW2010radialmasa}.

Now let $C$ be a masa in $Qz$ which contains $Bz$. Again by solidity and maximal injectivity, $C\subset Az$. Since $Qz$ is of type II$_1$, there exists two non-zero projections $p_1,p_2\in C$ and a partial isometry $v\in Qz$, such that $vv^*=p_1, v^*v=p_2, p_1p_2=0$. Then we have $E_A(v)=E_A(p_1vp_2)=p_1E_A(v)p_2=0$ so that $vu\perp uv$. However we also know that $vu=uv$, hence $v=0$. This contradicts that $p_1,p_2\neq 0$.

Thus, $Q$ has to be of type I. Let $C$ be a masa in $Q$ containing $B$. Again $C\subset A$. By Kadison's result \cite{kadison84diagonal}, $C$ is regular in $Q$. Both $A$ and $Q$ lie in the normalizer of $C$, so
they together generate an amenable algebra containing $A$. By maximal amenability of $A$, it follows that $Q\subset A$.
\end{proof}

Thus, in order to confirm Peterson's conjecture for the radial masa, it suffices to prove the strong-AOP as in Proposition~\ref{prop: s-AOP implies disjointness} for the radial masa. This section is mainly devoted to establish this fact.

Let $\Gamma=\mathbb{F}_{N+1}, N\in \mathbb{N}$. Write $K:=2N+1$ for later use. Denote by $\omega_n=\sum_{g\in G,|g|=n}u_g$, for $n=1,2,3,\cdots$ and let $\omega_0=u_e$. Let $M=L(\Gamma)$ be the free group factor and let $C=\{\omega_1\}''\subset M$ be the radial masa. $\{\omega_n\}_{n\geq 0}$ forms an orthogonal basis for $L^2(C)$. 

Let $\mathcal{K}_i$ be the finite-dimensional subspace of $\mathcal{H}:=L^2(M)$ spanned by all words of length $i$ and we denote by $Q_i$ the orthogonal projection from $\mathcal{H}$ onto $\mathcal{K}_i$. For $\xi\in \mathcal{K}_i$ and $n,m\in \mathbb{N}\bigcup \{0\}$, we define the following
\[\xi_{n,m}:=\dfrac{Q_{i+m+n}(\omega_n \xi \omega_m)}{K^{(n+m)/2}}.\]

R{\u a}dulescu \cite{radulescu91radialmasa} discovered that there is a nice decomposition of $\mathcal{H}\ominus L^2(C)=\bigoplus_{i\geq 1}\mathcal{H}_i$ into a direct sum of $C$-$C$-bimodules, each $\mathcal{H}_i$ has a distinguished unit vector $\xi^i$, which is from $\mathcal{K}_{l(i)}$, for some $l(i)\in \mathbb{N}$, such that $\mathcal{H}_i$ is generated by $\xi^i$ as a $C$-$C$-bimodule. 

Moreover, by \cite[Lemma 3, Lemma 6]{radulescu91radialmasa}, for those $i$ with $l(i)\geq 2$, we have that $\{\xi^i_{n,m}\}_{n,m\geq 0}$ forms an orthonormal basis for $\mathcal{H}_i$. For those $i$ with $l(i)=1$ (there are finitely many such $i$'s), $\{\xi^i_{n,m}\}_{n,m\geq 0}$ is no longer an orthonormal basis for $\mathcal{H}_i$, however  for any $i,j\geq 1$, the linear mapping $T_{i,j}:\mathcal{H}_i\rightarrow\mathcal{H}_j$, given by 
\[T_{i,j}(\xi^i_{n,m})=\xi^j_{n,m},\]
extends uniquely to an invertible bounded operator. Furthermore, there is a universal constant $C_1>0$ such that \[\|T_{i,j}^{\pm 1}\|\leq C_1,\forall i,j\geq 1.\]
\begin{rem}\label{rem: Riesz basis}
Recall that in a separable Hilbert space, a sequence of vectors $\{v_n\}$ forms a \textit{Riesz basis} (for the basics of Riesz basis, see, e.g.\ \cite{rieszbasis}), if  $\{v_i\}$ is the image of some orthonormal basis under some bounded invertible operator. It is also equivalent to the fact that there exists some $A,B>0$ such that for any $(c_n)\in \ell^2$, $A\sum |c_n|^2\leq \norm{\sum c_nv_n}\leq B\sum |c_n|^2$. In this case, every vector $x$ in the Hilbert space has a \textit{unique} decomposition $x=\sum c_nv_n$, for some $(c_n)\in \ell^2$. It follows that  $\left\{\xi^i_{n,m}\right\}_{i\geq 1,n,m\geq 0}$ forms a Riesz basis for $L^2(M)\ominus L^2(C)$. Consequently, for any $x\in L^2(M)\ominus L^2(C)$, there is a \textit{unique} decomposition $x=\sum_{i\geq 1,n,m\geq 0}a^i_{n,m}\xi^i_{n,m}$ for some $\left(a^i_{n,m}\right)_{n,m,i}\in \ell^2$. We call $\{\xi^i_{n,m}\}_{i\geq 1,n,m\geq 0}$ the \textit{R{\u a}dulescu basis} for $L^2(M)\ominus L^2(C)$.

Sometimes it will be convenient to use the following convention: we write $\xi^i_{n,m}$ for all $n,m\in \mathbb{Z}$, where we define $\xi^i_{n,m}=0$ whenever $n<0$ or $m<0$.
\end{rem}

The key computation in \cite{CFRW2010radialmasa} is that when considering the AOP in the case of the radial masa, the R{\u a}dulescu basis plays the same role as the canonical basis for the generator masa case. However, in our approach, the R{\u a}dulescu basis suffers from a lack of right modularity. Instead,  $\{\omega_n \xi^i\omega_m\}$, after proper normalization, is the more natural basis to work with. 

We collect some relations between $\omega_n \xi^i\omega_m$ and $\xi^i_{n,m}$'s, due to R{\u a}dulescu, in the following lemma:
\begin{lem}[Lemma 2, 6 in \cite{radulescu91radialmasa}]\label{lem: relations between different bases}
The following statements hold for all $n,m\geq 0$:\\
(1) If $l(i)\geq 2$, then $\omega_n\xi^i\omega_m=K^\frac{n+m}{2}\xi^i_{n,m}-K^{\frac{n+m-2}{2}}\left(\xi^i_{n,m-2}+\xi^i_{n-2,m}\right)+K^{\frac{n+m-4}{2}}\xi^i_{n-2,m-2}$;\\
(2) If $l(i)=1$, then there is some $\sigma=\sigma(i)\in \{-1,1\}$ such that
\begin{equation*}
\begin{split}
\omega_n\xi^i\omega_m&=K^\frac{n+m}{2}\xi^i_{n,m}-K^{\frac{n+m-2}{2}}\left(\xi^i_{n,m-2}+\xi^i_{n-2,m}+\sigma\xi^i_{n-1,m-1}\right)\\
&+\sum_{k\geq 2}(-\sigma)^k K^{\frac{n+m-2k}{2}}\left(\sigma \xi^i_{n-k-1,m-k+1}+\sigma \xi^i_{n-k+1,m-k-1}+2\xi^i_{n-k,m-k}\right);
\end{split}
\end{equation*}
(3) For all $i,j\geq 1$, the linear mapping $S_{i,j}:\mathcal{H}_i\rightarrow\mathcal{H}_j$ given by 
\[S_{i,j}\left(\omega_n \xi^i\omega_m\right)=\omega_n \xi^j\omega_m, \forall n,m\geq 0,\]
is well-defined and extends to an invertible bounded operator between the two subspaces, with $\sup_{i,j}\|S_{i,j}^{\pm 1}\|\leq C_2,$
for some uniform constants $0<C_2<\infty.$
 
\end{lem}

\begin{lem}\label{lem: verifying new basis}
 $\left\{\eta^i_{n,m}:=\dfrac{\omega_n\xi^i\omega_m}{K^{(n+m)/2}}\right\}_{i\geq 1,n,m\geq 0}$ forms a Riesz basis for $L^2(M)\ominus L^2(C)$. 

Therefore, for any $x\in L^2(M)\ominus L^2(C)$, there is a unique decomposition $x=\sum_{i\geq 1, n,m\geq 0} b_{n,m}^i\eta^i_{n,m}$ where $\left(b^i_{n,m}\right)_{i\geq 1,n,m\geq 0}\in \ell^2$.

\end{lem}
\begin{proof}
By (3) of the previous lemma, it suffices to prove the conclusion for some fixed $i\geq 1$ with $l(i)\geq 2$.

Fix $i\geq 1$ with $l(i)\geq 2$ and $(a_{n,m})_{n,m}\in \ell^2$. We will omit the superscript $i$, since no confusion will appear. Using part (1) of the previous lemma, we have 
\begin{equation*}
\begin{split}
\sum_{n,m\geq 0}a_{n,m}\eta_{n,m}&=\sum_{n,m\geq 0}a_{n,m}\left(\xi_{n,m}-\dfrac{\xi_{n,m-2}}{K}-\dfrac{\xi_{n-2,m}}{K}+\dfrac{\xi_{n-2,m-2}}{K^2}\right)\\
&=\sum_{n,m}\left(a_{n,m}-\dfrac{a_{n,m+2}}{K}-\dfrac{a_{n+2,m}}{K}+\dfrac{a_{n+2,m+2}}{K^2}\right)\xi_{n,m}\\
&=\sum_{n,m}\left(\left(a_{n,m}-\dfrac{a_{n,m+2}}{K}\right)-\dfrac{1}{K}\left(a_{n+2,m}-\dfrac{a_{n+2,m+2}}{K}\right)\right)\xi_{n,m},
\end{split}
\end{equation*}
hence by repeated use of the triangle inequality, we have 
\begin{equation*}
\begin{split}
\norm{\sum a_{n,m}\eta_{n,m}}_2&=\left(\sum_{n,m\geq 0} \left|\left(\left(a_{n,m}-\dfrac{a_{n,m+2}}{K}\right)-\dfrac{1}{K}\left(a_{n+2,m}-\dfrac{a_{n+2,m+2}}{K}\right)\right)\right|^2\right)^{1/2}\\
&\geq \left(\sum_{n,m\geq 0}\left|a_{n,m}-\dfrac{a_{n,m+2}}{K}\right|^2\right)^{1/2}-\dfrac{1}{K}\left(\sum_{n,m\geq 0}\left|a_{n+2,m}-\dfrac{a_{n+2,m+2}}{K}\right|^2\right)^{1/2}\\
&\geq \left(1-\dfrac{1}{K}\right)\left(\sum_{n,m\geq 0}\left|a_{n,m}-\dfrac{a_{n,m+2}}{K}\right|^2\right)^{1/2}\\
&\geq \left(1-\dfrac{1}{K}\right)^2\left(\sum_{n,m\geq 0}\left|a_{n,m}\right|^2\right)^{1/2}.
\end{split}
\end{equation*}
The other side of the inequality is easy, since each $a_{n,m}$ only appears at most four times. Thus there is a $B>0$, such that 
\[\norm{\sum a_{n,m}\eta_{n,m}}_2^2\leq B\sum\left|a_{n,m}\right|^2,\]
 So we are done.
\end{proof}

\begin{rem}
Because both $\|T_{i,j}^{\pm 1}\|$ and $\|S_{i,j}^{\pm 1}\|$ are uniformly bounded, there is a $C_0>0$ such that $\|T_{i,j}^{\pm 1}\|\leq C_0,\|S_{i,j}^{\pm 1}\|\leq C_0$, and for any $\left(c^i_{n,m}\right)\in \ell^2$, 
\begin{equation*}
\begin{split}
\dfrac{1}{C_0}\sum_{i,n,m}\left|c^i_{n,m}\right|^2\leq \|\sum_{n,m\geq 0,i\geq 1}c^i_{n,m}\xi^i_{n,m}\|^2_2\leq C_0\sum_{i,n,m}\left|c^i_{n,m}\right|^2,\\
\dfrac{1}{C_0}\sum_{i,n,m}\left|c^i_{n,m}\right|^2\leq \|\sum_{n,m\geq 0,i\geq 1}c^i_{n,m}\eta^i_{n,m}\|^2_2\leq C_0\sum_{i,n,m}\left|c^i_{n,m}\right|^2
\end{split}
\end{equation*}
\end{rem}

For each $k\in \mathbb{N}$, define $L_k, L_k': L^2(M)\ominus L^2(C)\rightarrow L^2(M)\ominus L^2(C)$ as follows
\begin{equation*}
\begin{split}
L_k\left(\sum_{i\geq 1,n,m\geq 0}a^i_{n,m}\xi^i_{n,m}\right):=\sum_{i\geq 1,n\leq k,m\geq 0}a^i_{n,m}\xi^i_{n,m},\\
L_k'\left(\sum_{i\geq 1,n,m\geq 0}b^i_{n,m}\eta^i_{n,m}\right):=\sum_{i\geq 1,n\leq k,m\geq 0}b^i_{n,m}\eta^i_{n,m},
\end{split}
\end{equation*}

i.e. $L_k$ (resp. $L_k'$) is the left ``projection'' onto the span of $\left\{\xi^i_{n,m}\right\}_{i,n,m}$ (resp.\ $\left\{\eta^i_{n,m}\right\}_{i,n,m}$) with the first subscript no larger than $k$. However one should be warned that they are merely idempotents, instead of projections, due to the presence of those $i$'s with $l(i)=1$. We can also define $R_k, R'_{k}$ for the right ``projections'' in the similar fashion. All these idempotents are bounded operators.
Let $L_{k}\vee R_{k}:=L_k+R_k-L_kR_k, L_{k}'\vee R_{k}':=L_k'+R_k'-L_k'R_k'$ .

\begin{lem}
$L_k'$ is right $C$-modular, $\forall k\geq 0.$ 
\end{lem}

\begin{proof}
Since $\{\omega_n\}_{n\geq 0}$ forms an orthogonal basis for $L^2(C)$ and $\{\eta_{n,m}^i\}_{i\geq 1, n,m\geq 0}$ is a Riesz basis for $L^2(M)\ominus L^2(C)$, it is sufficient to show that  $L_k'(\eta^i_{n,m}\omega_l)=L_k'(\eta^i_{n,m})\omega_l$.

 The definition of the $\eta_{n,m}^i$'s clearly implies that $\eta^i_{n,m}\omega_l\in$ span$\{\eta_{n,k}^i\}_{k\geq 0}$, that is, multiplying $\omega_l$ on the right does not change neither the upper nor left index of $\eta_{n,m}^i$, thus $L_k'(\eta^i_{n,m}\omega_l)=L_k'(\eta^i_{n,m})\omega_l$ and the proof is complete.
\end{proof}

We will need the following result from \cite{CFRW2010radialmasa}:

\begin{lem}[Lemma 4.3, Theorem 6.2 in \cite{CFRW2010radialmasa}]\label{lem: combinatorics from CFRW}
Given $(x_k)_k\in M^{\omega}\ominus C^{\omega}$, if for every $k_0\in \mathbb{N}$, we have that $\lim_{k\rightarrow\omega}\norm{(L_{k_0}\vee R_{k_0}) \left(x_k\right)}_2=0$, then for any $y_1,y_2\in L^2(M)\ominus L^2(C)$, $y_1(x_k)_k\perp (x_k)_ky_2.$

\end{lem}

Now we state the key technical result of this paper.

\begin{prop}\label{prop: s-AOP radial masa}
Let $\Gamma=\mathbb{F}_{N+1}$ be a non-abelian free group with finitely many generators and $M=L(\Gamma)$ the corresponding group von Neumann algebra. Denote by $C$ the radial masa of $M$ and suppose that $B\subset C$ is a diffuse von Neumann subalgebra. Then for any $(x_k)_k\in B'\bigcap M^{\omega}\ominus C^{\omega}$ and $y_1,y_2\in M\ominus C$, where $\omega$ is a free ultrafilter, we have that $y_1(x_k)_k\perp (x_k)_ky_2$ in $L^2(M^{\omega})$.
\end{prop}

We will break the proof into several lemmas.

Let $(x_k)_k\in B'\bigcap M^{\omega}\ominus C^{\omega}$ and $y_1,y_2\in M\ominus C$ be given. For each $k$, we can assume $x_k\in M\ominus C\subset L^2(M)\ominus L^2(C), ||x_k||\leq 1$ and write its decompositions with respect to $\left\{\xi^i_{n,m}\right\}_{i\geq 1,n,m\geq 0}$ and $\left\{\eta^i_{n,m}\right\}_{i\geq 1,n,m\geq 0}$, respectively:
\[x_k=\sum_{i\geq 1,n,m\geq 0}a^{i,k}_{n,m}\xi^i_{n,m}=\sum_{i\geq 1,n,m\geq 0}b^{i,k}_{n,m}\eta^i_{n,m}, \]
where both $\left(a^{i,k}_{n,m}\right)_{i\geq 1,n,m\geq 0}$ and $\left(b^{i,k}_{n,m}\right)_{i\geq 1,n,m\geq 0}$ are from $\ell^2$.

 Since $B$ is diffuse, we can choose a sequence $\{u_k\}_k$ in the unitary group of $B$, which converges to 0 weakly. Recall that $\left\{\dfrac{\omega_i}{||\omega_i||_2}\right\}_{i\geq 0}$ is an orthonormal basis for $L^2(C)$. Moreover, for any fixed $N_0\geq 0$, $\omega_n \omega_m$ will be supported on those $\omega_i$'s with $i>N_0$, provided that $|m-n|> N_0$. We first need to approximate each $u_k$ using finite linear combinations of $\omega_i$'s.

\begin{lem}\label{lem: approx unitaries 2}

 For each fixed $N_0$, there exists a sequence $\{S_k\}_{k\geq 1}$ of non-empty, disjoint, finite subsets of $\mathbb{N}\cup \{0\}$ and a sequence of strictly increasing natural numbers $\{n_k\}_{k\geq 1}$, such that in the decomposition with respect to $\{\omega_i\}_{i\geq 0}$, the supports of elements from $\{\omega_m\omega_n:m\in S_i, n\leq N_0 \}$ and the supports of elements from $\{\omega_m\omega_n:m\in S_j, n\leq N_0 \}$ are disjoint, whenever $i,j\geq 1, i\neq j$. Moreover, there exists a sequence  $\{v_k\}_k$ in $C$, with $v_k\in span\{\omega_i:i\in S_k\}$ such that $||v_k||\leq 2$ and $||v_k-u_{n_k}||_2\leq \dfrac{1}{2^k}$. 
 
 Moreover, one can construct $\{v_k\},\{S_k\}$ such that there is a sequence $\{F_k\}$ of strictly increasing natural numbers such that $L'_{N_0}\left(v_i x\right)=L'_{N_0}\left(v_i\left(L'_{F_{i+1}}-L'_{F_i}\right)(x)\right)$, for all $x\in L^2(M)\ominus L^2(C)$.

\end{lem}

\begin{proof}
Throughout this lemma, for any $x\in C$, we always consider the Fourier expansion of $x$ with respect to $\{\omega_i\}_{i\geq 0}$. Moreover, if $x=\sum_{i\geq 0}a_i\omega_i$ and $F\subset \mathbb{N}\cup\{0\}$, we will use the notation $P_F(x):=\sum_{i\in F}a_i\omega_i$. 

We construct $\{S_k\}$, $\{n_k\}$ and $\{v_k\}$ inductively. Since $span\{\omega_n\}_{n\geq 0}$ is a weakly dense *-subalgebra of $C$, Kaplansky Density Theorem implies that there exists a sequence $\{z_k\}_k$ of elements  in $C$, whose Fourier expansions are finitely supported, such that $||z_k||\leq 3/2$ and $||u_k-z_k||_2\leq \dfrac{1}{4^k}$. For each $k$, suppose that $z_k$ is supported on $\{\omega_i\}_{i\in T_k}$, where $T_k \subset \mathbb{N}\cup\{0\}$ is some finite subset. Let $n_1=1$, $v_1=z_1$ and $S_1=T_1$. Then $||v_1||\leq 2$ and $||v_1-u_{n_1}||_2\leq 1/2$ and $v_1\in span\{\omega_i:i\in S_1\}$.  

Now suppose that $S_1,\cdots,S_k$ and $n_1,\cdots,n_k$ have already been chosen. Then there exists a finite subset  $F_{k+1}\subset \mathbb{N}\cup \{0\}$, such that $\bigcup_{1\leq i\leq k}S_i\subset F_{k+1}$ and for any $S\subset \mathbb{N}\cup \{0\}\backslash F_{k+1}$, we always have that in the decomposition with respect to $\{\omega_i\}_{i\geq 0}$, the supports of elements from $\{\omega_m\omega_n:m\in \cup_{1\leq i\leq k}S_i, n\leq N_0 \}$ and the supports of elements from $\{\omega_m\omega_n:m\in S, n\leq N_0 \}$ are disjoint (for example, one can let $F_{k+1}=\{0,1,\cdots, max\cup_{1\leq i\leq k}S_i+3N_0 \}$). Now since $u_k\rightarrow 0$ weakly, there is a natural number $n_{k+1}>n_k$, such that with respect to the basis $\{\omega_i\}_{i\geq 0}$, the Fourier coefficient of $z_{n_{k+1}}$ has absolute value less than $\dfrac{1}{4^k|F_{k+1}|||\omega_i||}$, for each $0\leq i\leq F_{k+1}$. Let $S_{k+1}:= T_{n_{k+1}} \backslash F_{k+1}, v_{k+1}:=P_{S_{k+1}}(z_{n_{k+1}})$. Then 
\begin{align*}
\|v_{k+1}-u_{n_{k+1}}\|_2&\leq \norm{v_{k+1}-z_{n_{k+1}}}_2+\norm{z_{n_{k+1}}-u_{n_{k+1}}}_2\\
&=\|P_{(T_{n_{k+1}}\backslash S_{k+1})}(z_{n_{k+1}})\|_2+\norm{z_{n_{k+1}}-u_{n_{k+1}}}_2\\
&= \norm{P_{F_{k+1}}(z_{n_{k+1}})}_2+\norm{z_{n_{k+1}}-u_{n_{k+1}}}_2\\
&\leq \left(\dfrac{\left|F_{k+1}\right|}{2^{2k}\left|F_{k+1}\right|}\right)^{1/2}+\dfrac{1}{4^{k+1}}\\
&\leq \dfrac{1}{2^{k+1}},
\end{align*}
and an easy estimate of the $\ell^1$-norm gives us 
\begin{align*}
\norm{v_{k+1}}\leq \norm{z_{k+1}}+\dfrac{|F_{k+1}|}{2^k|F_{k+1}|}\leq 2.
\end{align*}
The last statement can be achieved by letting the supports of $\{v_k\}_k$ mutually far away. For example, choose the gap between $S_i$ and $S_j$ to be greater than $3N_0$ and let $F_k$ be the collection of elements of $\mathbb{N}\cup\{0\}$ between $\min_{n\in S_k}|n|-N_0$ and $\max_{n\in S_k} |n|+N_0$. 
\end{proof}

Thus by taking a subsequence if necessary, we may assume that $\{v_k\}$ is a sequence in $C$, such that $v_k\in span\{\omega_i:i\in S_k\}$ for some finite subset $S_k\subset \mathbb{N}$, $||v_k||\leq 2$, $||v_k-u_{k}||_2\leq \dfrac{1}{2^k}$ and $v_i\omega_k\perp v_j\omega_k$, for all $i,j\geq 1, i\neq j$ and all $0\leq k\leq N_0$ and there is a sequence $\{F_k\}$ of strictly increasing natural numbers such that $L'_{N_0}\left(v_i x\right)=L'_{N_0}\left(v_i\left(L'_{F_{i+1}}-L'_{F_i}\right)(x)\right)$, for all $x\in L^2(M)\ominus L^2(C)$.
 
\begin{lem}\label{lem: step1}
 $\lim_{k\rightarrow\omega}\norm{L'_{N_0}({x_k})}_2= 0$.
\end{lem}
 
\begin{proof}

Fix a small $\epsilon>0$. First choose some large $N_1<N_2$ such that $2\sum_{i=N_1}^{N_2}\norm{v_i-u_i}_2^2\leq \epsilon$ and $\dfrac{4\norm{L'_{N_0}}^2C_0^2+1}{N_2-N_1}\leq \epsilon$. Then we have 
\begin{equation*}
\begin{split}
\lim_{k\rightarrow\omega}\sum_{i=N_1}^{N_2}\left\langle L'_{N_0}(v_i x_k),L'_{N_0}(v_ix_k)\right\rangle 
&\geq\lim_{k\rightarrow\omega}\sum_{i=N_1}^{N_2}\left\langle L'_{N_0}(u_i x_k), L'_{N_0}(u_ix_k)\right\rangle-\epsilon \\
&=\lim_{k\rightarrow\omega}\sum_{i=N_1}^{N_2}\left\langle L'_{N_0}(x_ku_i), L'_{N_0}(x_ku_i)\right\rangle-\epsilon \\
&= \lim_{k\rightarrow\omega}\sum_{i=N_1}^{N_2}\left\langle L'_{N_0}(x_k)u_i, L'_{N_0}(x_k)u_i\right\rangle-\epsilon \\
&=(N_2-N_1)\lim_{k\rightarrow\omega}\norm{L'_{N_0}(x_k)}_2^2-\epsilon.
\end{split}
\end{equation*}
 The second line uses the assumption that $(x_k)_k\in B'\cap M^{\omega}$ and the third line uses the fact that $L'_{N_0}$ is a right-$C$ modular map, i.e.
\[L'_{N_0}(xa)=L'_{N_0}(x)a, \forall x\in L^2(M)\ominus L^2(C), \forall a\in C.\]

Meanwhile,
\begin{equation*}
\begin{split}
\sum_{i=N_1}^{N_2}\left\langle L'_{N_0}\left(v_i x_k\right),L'_{N_0}(v_ix_k)\right\rangle 
&=\sum_{i=N_1}^{N_2}\left\langle L'_{N_0}\left(v_i \left(L'_{F_{i+1}}-L'_{F_i}\right)(x_k)\right),L'_{N_0}\left(v_i\left(L'_{F_{i+1}}-L'_{F_i}\right)(x_k)\right)\right\rangle\\
&\leq \norm{L'_{N_0}}^2\sum_{i=N_1}^{N_2}\left\langle v_i \left(L'_{F_{i+1}}-L'_{F_i}\right)(x_k),v_i\left(L'_{F_{i+1}}-L'_{F_i}\right)(x_k)\right\rangle\\
&\leq \sum_{i=N_1}^{N_2}\norm{L'_{N_0}}^2\norm{v_i}^2\norm{\left(L'_{F_{i+1}}-L'_{F_i}\right)(x_k)}_2^2\\
&\leq 4\norm{L'_{N_0}}^2\sum_{i=N_1}^{N_2}C_0\sum_{j\geq 1,F_{i}+1\leq n\leq F_{i+1},m\geq 0}\left|b^{j,k}_{n,m}\right|^2\\
&\leq 4\norm{L'_{N_0}}^2 C_0 \sum_{j\geq 1,0\leq  n\leq F_{N_2},m\geq 0}\left|b^{j,k}_{n,m}\right|^2\\
&\leq 4\norm{L'_{N_0}}^2C_0^2\norm{x_k}_2^2\leq 4\norm{L'_{N_0}}^2C_0^2.
\end{split}
\end{equation*}
Therefore, we conclude that $\lim_{k\rightarrow\omega}\norm{L'_{N_0}(x_k)}_2^2\leq\dfrac{ 4||L'_{N_0}||^2C_0^2+1}{N_2-N_1}\leq \epsilon$ can be made arbitrarily small. Thus the proof for Lemma~\ref{lem: step1} is complete.
\end{proof}

\begin{lem}\label{lem: step2}
 $\lim_{k\rightarrow\omega}\norm{L_{N_0}({x_k})}_2= 0$.
\end{lem}
\begin{proof}

 We use the relations between $\eta^i_{n,m}$ and $\xi^i_{n,m}$, as stated above in Lemma~\ref{lem: relations between different bases} and Lemma~\ref{lem: verifying new basis}.

First, since $x_k=\sum_{i\geq 1,l(i)\geq 2,n,m\geq 0}a^{i,k}_{n,m}\xi^i_{n,m}\oplus \sum_{i\geq 1,l(i)=1,n,m\geq 0}a^{i,k}_{n,m}\xi^i_{n,m}$, it suffices to consider separately the part with $i$'s such that $l(i)\geq 2$ and the part with $i$'s such that $l(i)=1$. 

For the $i$'s with $l(i)\geq 2$, recall that $\eta^i_{n,m}=\xi^i_{n,m}-K^{-1}\left(\xi^i_{n,m-2}+\xi^i_{n-2,m}\right)+K^{-2}\xi^i_{n-2,m-2}$ so that 
$a^i_{n,m}=b^i_{n,m}-\dfrac{b^i_{n,m+2}}{K}-\dfrac{b^i_{n+2,m}}{K}+\dfrac{b^i_{n+2,m+2}}{K^2}$. Therefore

\begin{equation*}
\begin{split}
\norm{L_{N_0}\left(\sum_{i\geq 1,l(i)\geq 2,n,m\geq 0}a^{i,k}_{n,m}\xi^{i,k}_{n,m}\right)}_2^2
&=\sum_{i\geq 1,l(i)\geq 2,N_0\geq n\geq 0,m\geq 0}\left|a^{i,k}_{n,m}\right|^2\\
&=\sum_{i\geq 1,l(i)\geq 2,N_0\geq n\geq 0,m\geq 0}\left|b^{i,k}_{n,m}-\dfrac{b^{i,k}_{n,m+2}}{K}-\dfrac{b^{i,k}_{n+2,m}}{K}+\dfrac{b^{i,k}_{n+2,m+2}}{K^2}\right|^2\\
&\leq 16\sum_{i\geq 1,l(i)\geq 2,N_0+2\geq n\geq 0,m\geq 0}\left|b^{i,k}_{n,m}\right|^2\\
&\leq 16 C_0\norm{L'_{N_0+2}\left(\sum_{i\geq 1,l(i)\geq 2,n,m\geq 0}b^{i,k}_{n,m}\eta^{i,k}_{n,m}\right)}_2^2,\\
\end{split}
\end{equation*}
and the last term goes to 0 as $k\rightarrow \omega$, by the previous lemma.

Now consider the $i$'s with $l(i)=1$. As there are only finitely many such $i$'s, we may restrict our attention to a single fixed $i$.

For some $\sigma \in \{1,-1\}$, we have that 
\begin{equation*}
\begin{split}
\sum_{n,m\geq 0}b^{i,k}_{n,m}\eta^{i,k}_{n,m}
&=\sum_{n,m}b^{i,k}_{n,m}\Bigg(\xi^{i,k}_{n,m}-\dfrac{\xi^{i,k}_{n-2,m}}{K}-\dfrac{\xi^{i,k}_{n,m-2}}{K}+\sigma \dfrac{\xi^{i,k}_{n-1,m-1}}{K}\\
&+\sum_{l\geq 2}\dfrac{(-\sigma)^l}{K^l}\left(\sigma \xi^{i,k}_{n-l-1,m-l+1}+\sigma\xi^{i,k}_{n-l+1,m-l-1}+2\xi^{i,k}_{n-l,m-l}\right)\Bigg)\\
&=\sum_{n,m}\Bigg(b^{i,k}_{n,m}-\dfrac{b^{i,k}_{n+2,m}}{K}-\dfrac{b^{i,k}_{n,m+2}}{K}+\dfrac{\sigma b^{i,k}_{n+1,m+1}}{K}\\
&+\sum_{l\geq 2}\dfrac{(-\sigma)^l}{K^l}\left(\sigma b^{i,k}_{n+l+1,m+l-1}+\sigma b^{i,k}_{n+l-1,m+l+1}+2b^{i,k}_{n+l,m+l}\right)\Bigg)\xi^{i,k}_{n,m}.
\end{split}
\end{equation*}

Therefore, for any fixed  $\epsilon>0, N_0\geq 0$, we find a large integer $N_1\gg N_0$, to be specified later, and we let $K_0=N_1-N_0$. By the triangle inequality, 
\begin{equation*}
\begin{split}
\left(\sum_{n\leq N_0,m\geq 0}|a^{i,k}_{n,m}|^2\right)^{1/2}
&\leq \left(\sum_{n\leq N_0,m\geq 0} \left|b^{i,k}_{n,m}-\dfrac{b^{i,k}_{n+2,m}}{K}-\dfrac{b^{i,k}_{n,m+2}}{K}+\dfrac{\sigma b^{i,k}_{n+1,m+1}}{K}\right|^2\right)^{1/2}\\
&+\sum_{2\leq l\leq K_0}\dfrac{1}{K^l}\left(\sum_{n\leq N_0,m\geq 0} \left|\sigma b^{i,k}_{n+l+1,m+l-1}+\sigma b^{i,k}_{n+k-1,m+k+1}+2b^{i,k}_{n+k,m+k}\right|^2\right)^{1/2}\\
&+\sum_{ l\geq  K_0+1}\dfrac{1}{K^l}\left(\sum_{n\leq N_0,m\geq 0} \left|\sigma b^{i,k}_{n+l+1,m+l-1}+\sigma b^{i,k}_{n+l-1,m+l+1}+2b^{i,k}_{n+l,m+l}\right|^2\right)^{1/2}.
\end{split}
\end{equation*}
We estimate the third term in the above inequality first:
\begin{equation*}
\begin{split}
&\sum_{ l\geq  K_0+1}\dfrac{1}{K^l}\left(\sum_{n\leq N_0,m\geq 0} \left|\sigma b^{i,k}_{n+l+1,m+l-1}+\sigma b^{i,k}_{n+l-1,m+l+1}+2b^{i,k}_{n+l,m+l}\right|^2\right)^{1/2}\\
&\leq \sum_{ l\geq  K_0+1}\dfrac{1}{K^l}\left(\sum_{n,m\geq 0} \left|\sigma b^{i,k}_{n+l+1,m+l-1}+\sigma b^{i,k}_{n+l-1,m+l+1}+2b^{i,k}_{n+l,m+l}\right|^2\right)^{1/2}\\
&\leq \sum_{ l\geq  K_0+1}\dfrac{1}{K^l}4\left(\sum_{n,m\geq 0} \left|b^{i,k}_{n,m}\right|^2\right)^{1/2}\\
&\leq \sum_{ l\geq  K_0+1}\dfrac{1}{K^l}4C_0\norm{x_k}_2\leq \dfrac{4C_0}{K^{K_0}(K-1)},
\end{split}
\end{equation*}
hence we can choose $N_1$ large enough so that $K_0$ is large, such that the third term is less than $\epsilon/3$, for any $k$.

Now we estimate the first and the second terms. To this end, we choose a large $k_0=k_0(N_1,\epsilon)$, such that for any $k\geq k_0$, we have that $4C_0K_0\left(\sum_{m\geq 0, n\leq N_1+1}\left|b^{i,k}_{n,m}\right|^2\right)^{1/2}$ is less than $\epsilon/3$. Thus both the first and the second term can be bounded above by $\epsilon/3$. Combine all these pieces together, we conclude that 
\[\norm{L_{N_0}\left(\sum_{n, m\geq 0}a^{i,k}_{n,m}\xi^{i,k}_{n,m}\right)}_2\leq C_0\left(\sum_{n\leq N_0,m\geq 0}|a^{i,k}_{n,m}|^2\right)^{1/2}\leq C_0\epsilon,\]

when $k$ is close enough to $\omega$. Since $\epsilon>0$ is arbitrary, we are done.
\end{proof}

\begin{proof}[Proof of Proposition~\ref{prop: s-AOP radial masa}]
The same proof for Lemma~\ref{lem: step2} shows that $\lim_{k\rightarrow\omega}\norm{R_{N_0}({x_k})}_2= 0$. So Lemma~\ref{lem: combinatorics from CFRW} applies.
\end{proof}

\begin{rem}
In fact, the same conclusion as in Proposition~\ref{prop: s-AOP radial masa} holds, if we replace the assumption ``$B\subset C$ diffuse'' by ``$B\subset C^{\omega}$ diffuse''. 
\end{rem}

\begin{thm}\label{thm: disjointness for radial masa}
The radial masa satisfies Peterson's conjecture.
\end{thm}
\begin{proof}
It is shown in \cite{op10uniquecartan} that $L(\mathbb{F}_N), N\geq 2$ is strongly solid, and the fact that the radial masa is singular is shown in \cite{radulescu91radialmasa}(another proof can be found in \cite{allanroger03Laplacianmasa}). Therefore, Proposition~\ref{prop: s-AOP implies disjointness} and Proposition~\ref{prop: s-AOP radial masa} imply the result.
\end{proof}
\begin{rem}
One can also use \cite[Theorem 8.1]{houdayer14structurebogoljubov} and Proposition~\ref{prop: s-AOP radial masa} to conclude Theorem~\ref{thm: disjointness for radial masa}. We thank Boutonnet for pointing it out to us.
\end{rem}

In fact, one can state a more general structural result for the inclusion $C\subset L(\mathbb{F}_{N+1})$.
\begin{thm}\label{thm: gamma-stability for radial masa}
Let $M=L(\mathbb{F}_{N+1})$ be a free group factor with $1\leq N<\infty$ and let $C\subset M$ be the radial masa. If $Q\subset M$ is a von Neumann subalgebra that has a diffuse intersection with $C$, then there exists a sequence of central projections $e_n\in Z(Q), n\geq 0$ such that 
\begin{itemize}
\item $e_0Q\subset C$;
\item For all $n\geq 1$, $e_nQ$ is a non-amenable II$_1$ factor such that $e_n(Q'\cap M^{\omega})=e_n(Q'\cap M)$ is discrete and abelian (even contained in $C$).
\end{itemize}
\end{thm}
\begin{proof}
Let $e_0\in Z(Q)$ be the maximal projection such that $e_0Q$ is amenable. Then $Qe_0\oplus C(1-e_0)$ is amenable and has a diffuse intersection with $C$ so it is contained in $C$ by Theorem \ref{thm: disjointness for radial masa}.
Moreover, $Q(1-e_0)$ has a discrete center, by solidity of $M$. This gives a sequence of central projections $\{e_n\}_{n\geq 1}$ such that for all $n\geq 1$, $e_nQ$ is a non-amenable II$_1$ factor. \\
Now fix $n\geq 1$. By \cite[Lemma 2.7]{adrian15cartanfreeproducts}, one can find a central projection $e\in Z((e_nQ)'\cap e_nMe_n)$ such that 
\begin{itemize} 
\item $e((e_nQ)'\cap e_nMe_n)=e((e_nQ)'\cap (e_nMe_n)^\omega)$ is discrete;
\item $(e_n-e)((e_nQ)'\cap (e_nMe_n)^\omega)$ is diffuse.
\end{itemize}
By \cite[Proof of Theorem 4.3]{jessepeterson2009L2rigidity}, the fact that  $(e_n-e)((e_nQ)'\cap (e_nMe_n)^\omega)$ is diffuse implies that $(e_n-e)Q$ is amenable. Since $e_nQ$ has no direct summand, this forces $e=e_n$.

Finally, $(Q\cap C)'\cap M$ is amenable, again by solidity. As it contains $C$, it has to be equal to $C$. In particular $Q'\cap M\subset (Q\cap C)'\cap M\cap C$. So the last part of the theorem is true.
\end{proof}

\begin{rem}
In \cite[Theorem 3.1]{houdayer14gammastability}, Houdayer showed the general situation for free products of $\sigma$-finite von Neumann algebras, which contains the strong-AOP for the generator masa in a free group factor as a special case.
Also, the strong-AOP as in Proposition~\ref{prop: s-AOP radial masa} means that for any diffuse subalgebra $B$ of the radial masa $C$, the inclusion $C\subset M$ has the \textit{AOP relative to} $B$, in the sense of  \cite[Definition 5.1]{houdayer14structurebogoljubov}. The unique maximal injective extension for any diffuse subalgebra of the generator masa is first shown by Houdayer \cite[Theorem 4.1]{houdayer14gammastability}. A proof via the study of centralizers is obtained by Ozawa \cite{ozawa15centralizer}.
\end{rem}

\begin{rem}\label{rem:counterexamples1}
Note that the disjointness result as in Theorem \ref{thm: disjointness for radial masa} is not true for arbitrary maximal amenable masa of a II$_1$ factor. For instance, if the inclusion $A\subset M$ has some nice decomposition, then $A$ does not have the uniqueness property as the generator masa in the above corollary. We give some such examples:
\begin{itemize}
\item Let $M=A_1*_{A_0}A_2$ be the amalgamated free product with $A_i$ amenable, and  $A_0$ diffuse, $A_0\neq A_i$,  $i=1,2$, then $A_0$ can be contained in different maximal amenable subalgebras. 
\item Let $M_1,M_2$ both be the free group factor and $A_i\subset M_i$ the corresponding generator masa, $i=1,2$. Then 
$A=A_1\overline{\otimes}A_2$ is a maximal injective subalgebra inside $M=M_1\overline{\otimes}M_2$. However, many other injective subalgebras could contain the diffuse subalgebra $A_1\otimes 1$.
\item Let $\Lambda < \Gamma$ be a singular subgroup in the sense of Boutonnet and Carderi (\cite[Definition 1.2]{boutonnet14centralizer}) and suppose $\Gamma$ acts on a finite diffuse amenable von Neumann algebra $Q$. Then $Q\rtimes \Lambda$ is maximal injective inside $Q\rtimes \Gamma$, by \cite[Theorem 1.3]{boutonnet14centralizer}. However again there are lots of different injective subalgebras containing $Q$ but are not contained in $Q\rtimes \Lambda$.
\end{itemize}
\end{rem}

\begin{rem}\label{rem:counterexample2}

We would like to mention an example in the ultra-product setting. Let $A\subset M$ be a singular masa inside a separable II$_1$ factor. Then for any free ultrafilter $\omega$, $\mathcal{A}:=A^{\omega}$ is a maximal injective masa in $\mathcal{M}:=M^{\omega}$, a result due to Popa (\cite[Theorem 5.2.1]{popa14kadison-singer}). However, it is well known that any two separable abelian subalegebras in a ultraproduct of II$_1$ factors are unitarily conjugate (\cite[Lemma 7.1]{popa83orthogonal}). In particular, $A$ is both contained in a maximal injective masa and a maximal hyperfinite subfactor of $\mathcal{M}$.
\end{rem}

\small
\bibliographystyle{plain}
\bibliography{ref}

\end{document}